\documentclass[a4paper, 11pt]{amsart} 
\usepackage[
  margin=1.9cm,
  includefoot,
  footskip=30pt,
]{geometry}
\usepackage{amsmath,amssymb,amscd}
\usepackage{amsfonts}
\usepackage[english]{babel}
\usepackage[leqno]{amsmath}
\usepackage{amssymb,amsthm}
\usepackage{mathrsfs} 
\usepackage{amscd}
\usepackage{enumerate}
\usepackage{epsfig}
\usepackage{relsize}
\usepackage{layout}
\usepackage[usenames,dvipsnames]{xcolor}
\usepackage[backref=page]{hyperref}
\usepackage{tikz}
\hypersetup{
 colorlinks,
 citecolor=Green,
 linkcolor=Red,
 urlcolor=Blue}
\usepackage[matrix,arrow,tips,curve]{xy}
\input{xy}
\xyoption{all}
\definecolor{darkgreen}{rgb}{0.0, 0.7, 0.0}
\definecolor{purple}{rgb}{0.5, 0.0, 0.5}
\definecolor{red}{rgb}{0.8, 0.2, 0.0}

\newtheorem{thm}{Theorem}[section]
\newtheorem{bthm}{Theorem}

\newtheorem{lemma}[thm]{Lemma}

\numberwithin{equation}{section}
\theoremstyle{definition}

\theoremstyle{remark}

\def \P{\mathbb{P}}

\def \E{\mathcal E}

\def\O{\mathcal O}

\def\M0{\mathcal M^0}

\begin{document}

\title[On the existence of smooth varieties of any dimension carrying no Ulrich bundles]{On the existence of smooth varieties of any dimension carrying no Ulrich bundles for some very ample line bundle}

\author[A.F. Lopez]{Angelo Felice Lopez}

\address{\hskip -.43cm Angelo Felice Lopez, Dipartimento di Matematica e Fisica, Universit\`a di Roma
Tre, Largo San Leonardo Murialdo 1, 00146, Roma, Italy. e-mail {\tt angelo.lopez@uniroma3.it}}

\thanks{Author is partially supported by the GNSAGA group of INdAM.}


\begin{abstract} 
Using some recent examples of \cite{a} in dimension $2$, we prove that for any $n \ge 3$ there exist many smooth $n$-dimensional varieties $X$ with a very ample line bundle $H$, such that there is no Ulrich bundle for $(X,H)$. Moreover, we find many new examples in dimension $2$.
\end{abstract}

\maketitle

\section{Introduction}

In their seminal 2003 paper \cite{es}, Eisenbud, Schreyer and Weyman introduced Ulrich vector bundles in the realm of algebraic geometry. Let $X$ be a smooth irreducible variety of dimension $n \ge 1$ and let $H$ be a very ample line bundle on $X$. A vector bundle $\E$ on $X$ is Ulrich if $H^i(\E(-pH))=0$ for all $i \ge 0$ and $1 \le p \le n$. The authors of \cite{es} showed what important consequences on the geometry of $X$ could be deduced from an Ulrich bundle and posed the main question of their existence (in fact, for an Ulrich sheaf on any projective scheme). In the case of complete intersections, existence followed by previous results in algebra \cite{hub}. Since then, many papers have been written to show that Ulrich bundles exist on several classes of pairs $(X,H)$. Some of these existence results include curves (and any polarization), many surfaces (see for example \cite{b2, f, c1, c2, cohu, lo}), some threefolds, including Fano threefolds \cite{b2, cfk1, cfk2}, and other scattered examples of higher dimensional varieties (see \cite{cmr, es, b2, cmp} and references therein).
Moreover, several results (see for example \cite{lm, lms, lo2, ls}) confirmed that the existence of an Ulrich bundle gives a lot of insights on the geometry of $X$.

Despite of all this work, the main existence problem remained open until the 2026 breakthrough paper \cite{a}, where several examples of surface pairs $(S,H)$ not carrying Ulrich bundles were produced. Anghel used the following obstruction to the existence of an Ulrich bundle on a surface pair $(S,H)$ (see \cite[Prop.~2.1]{a}, improving on a previous idea of Beauville \cite{b}):
$$H^2 < K_S^2-8\chi(\O_S).$$
Our starting point was to notice that the above obstruction has a natural generalization in any dimension (see Lemma \ref{nonext}):
$$(n+1)H^n < \big(K_X^2-2c_2(X)\big)H^{n-2}.$$
Using the above and some examples present in \cite{a}, we produce natural examples of many $n$-dimensional smooth pairs $(X,H)$ that do not carry any Ulrich bundle. Our main result is the following

\begin{bthm}
\label{main}
Let $(S_k,L_k)$ be the pair consisting of a smooth irreducible surface and a very ample line bundle as in \cite[Thm.~1.2]{a}. Let $n \ge 3$, let $Y$ be any smooth irreducible $(n-2)$-dimensional variety and let $M$ be any very ample line bundle on $Y$. Let 
$$X=S_k \times Y, \ \ H=L_k \boxtimes M$$
and let $X' \in |aH|, a \ge 1$ be smooth, $H'=H_{|X'}$. 
Then for $k \gg 0$ the pairs $(X, H), (X', H')$, and all their successive hyperplane sections until dimension $2$, do not carry an Ulrich vector bundle. Moreover, $X'$ is not isomorphic to a (non-trivial) product.
\end{bthm}
Note that the surface sections of $X$ and $X'$ gives new examples of surface pairs, that is, different from the ones in \cite{a}, not carrying Ulrich bundles.

We observe that an explicit bound for $k$, depending on $(n, Y, M)$ (and $a$ in the second case), can be easily deduced from the proof of the theorem.

Given the above results, perhaps at this point the main question that remains open is the existence of Ulrich bundles for sufficiently ample polarizations. The latter is known for surfaces by \cite{cohu}.

\section{A non-product criterion}

We give a simple result that shows that a hypersurface section of a product is not a product.

\begin{lemma}
\label{nopro}
Let $X_1, X_2$ be two smooth irreducible varieties of dimension at least $2$ and let $H$ be a very ample line bundle on $X:=X_1 \times X_2$. Let $a \ge 1$ and let $D \in |aH|$ be smooth. Then $D$ is not isomorphic to a (nontrivial) product.
\end{lemma}
\begin{proof}
Consider, for a given smooth variety $V$, the $k$-th Newton class
$$p_k(V):=p_k(T_V)=k!ch_k(T_V).$$
For $i \in \{1,2\}$, let $n_i = \dim X_i \ge 2, n=n_1+n_2, n_i < n-1$ and let $\pi_i : X=X_1 \times X_2 \to X_i$ be the projections. Then $T_X=\pi_1^*T_{X_1} \oplus \pi_2^*T_{X_2}$, hence
\begin{equation}
\label{p1}
p_{n-1}(X)=\pi_1^*p_{n-1}(X_1)+\pi_2^*p_{n-1}(X_2)=0.
\end{equation}
Suppose that
$$D \cong Y_1 \times Y_2$$
for two smooth projective positive dimensional varieties $Y_i, i \in \{1,2\}$. Let $p_i : D \cong Y_1 \times Y_2 \to Y_i$ be the projections, so that $T_D \cong p_1^*T_{Y_1} \oplus p_2^*T_{Y_2}$. Since $\dim D=n-1=\dim Y_1+\dim Y_2$, we have that $\dim Y_i < n-1$, hence
\begin{equation}
\label{p2}
p_{n-1}(D)=p_1^*p_{n-1}(Y_1)+p_2^*p_{n-1}(Y_2)=0.
\end{equation}
On the other hand, the exact sequence
$$0 \to T_D \to  {T_X}_{|D} \to \O_D(aH) \to 0$$
and the additivity of Newton classes give, using \eqref{p1} and \eqref{p2}, the contradiction
$$0=p_{n-1}(D)=p_{n-1}(X)_{|D}-p_{n-1}(\O_D(aH))=a^nH^n.$$
\end{proof}

\section{An $n$-dimensional obstruction to the existence of an Ulrich bundle}

The following is a simple generalization of Anghel's \cite[Prop.~2.1]{a} (and Beauville's \cite{b}) obstruction.
\begin{lemma}
\label{nonext}
Let $X$ be a smooth irreducible variety of dimension $n \ge 2$, let $H$ be a very ample line bundle on $X$.
If
\begin{equation}
\label{dis}
(n+1)H^n < \big(K_X^2-2c_2(X)\big)H^{n-2}
\end{equation}
then there is no Ulrich bundle for $(X, H)$ and for all successive smooth hyperplane sections $(X_i,H_{|X_i}), 2 \le i \le n-1$, where $X_i = X \cap H_1 \cap \ldots \cap H_{n-i}$ with $H_j \in |H|$.
\end{lemma}
\begin{proof}
Assume that there exists a rank $r$ Ulrich bundle $\E$ for $(X, H)$. Set
$$\alpha = c_1(\E)-\frac{r}{2}\big(K_X+(n+1)H\big).$$
As is well known (see for example \cite[Lemma 3.1(iii)]{lr}), we have that $\alpha H^{n-1}=0$. Let $S$ be a smooth surface section of $X \subset \P H^0(H)$. Then $\alpha_{|S}H_{|S}=\alpha H^{n-1}=0$ and therefore $\alpha_{|S}^2 \le 0$ by the Hodge index theorem. Therefore
\begin{equation}
\label{1}
\alpha^2 H^{n-2} \le 0.
\end{equation}
By \cite[Lemma 3.1(viii)]{lr} we have that 
$$c_2(\E) H^{n-2}=\frac{1}{2}\big(c_1(\E)^2-c_1(\E) K_X\big)H^{n-2}+\frac{r}{12}\big(K_X^2+c_2(X)-\frac{3n^2+5n+2}{2}H^2\big)H^{n-2}$$
and therefore an easy calculation gives that
$$\Delta(\E)=\big(2rc_2(\E)-(r-1)c_1(\E)^2\big)H^{n-2}=\alpha^2 H^{n-2}+\frac{r^2}{12}\big(2c_2(X)-K_X^2+(n+1)H^2\big)H^{n-2}.$$
Since $\E$ is $\mu_H$-semistable by \cite[Thm.~2.9(a)]{ch}, it follows by Bogomolov's inequality that $
\Delta(\E) \ge 0$. Using \eqref{1} we therefore get that
$$\big(2c_2(X)-K_X^2+(n+1)H^2\big)H^{n-2} \ge 0$$
that is
$$(n+1)H^n \ge (K_X^2-2c_2(X))H^{n-2}$$
a contradiction. This proves that if \eqref{dis} holds, then then there is no Ulrich bundle for $(X, H)$. 

To prove that the same property holds on $(X_i,H_{|X_i})$, it is enough to show that \eqref{dis} is equivalent to the corresponding one on a smooth hyperplane section of $X$. Set $X'=X_{n-1}$, so that $X' \in |H|$, and $H'=H_{|X'}$. We have
\begin{equation}
\label{111}
(H')^{n-1}=H^n
\end{equation}
and
\begin{equation}
\label{121}
K_{X'}^2(H')^{n-3}=(K_X+H)^2H^{n-2}=K_X^2H^{n-2}+2K_XH^{n-1}+H^n.
\end{equation}
From the exact sequence
$$0 \to T_{X'} \to {T_X}_{|X'} \to \O_{X'}(H') \to 0$$
we obtain $c_2(X')=\big(c_2(X)+K_XH+H^2\big)_{|X'}$, and consequently,
\begin{equation}
\label{131}
c_2(X')(H')^{n-3}=c_2(X)H^{n-2}+K_XH^{n-1}+H^n.
\end{equation}
Now, using \eqref{111}, \eqref{121} and \eqref{131}, we see that \eqref{dis} for $(X', H')$, that is
$$n(H')^{n-1}<(K_{X'}^2-2c_2(X'))(H')^{n-3}$$
is equivalent to
$$nH^n<K_X^2H^{n-2}-2c_2(X)H^{n-2}-H^n$$
hence to
$$(n+1)H^n<\big(K_X^2-2c_2(X)\big)H^{n-2}$$
that is to \eqref{dis} for $(X,H)$.
\end{proof}

\section{Proof of the main theorem}

\renewcommand{\proofname}{Proof of Theorem \ref{main}}
\begin{proof}
We set $S=S_k$ and $L=L_k$ and let $p: X=S \times Y \to S$ and $q: X=S \times Y \to Y$ be the two projections. 

We compute the quantities in \eqref{dis}. 

We have
$$H^n=(p^*L+q^*M)^n=\sum\limits_{j=0}^n \binom{n}{j}p^*L^jq^*M^{n-j}$$
and therefore
\begin{equation}
\label{h}
H^n=\binom{n}{2}L^2 M^{n-2}.
\end{equation}
Similarly, if $n \ge 4$ we get
\begin{equation}
\label{hn-2}
\begin{aligned}
H^{n-2} & = \big(p^*L+q^*M\big)^{n-2}=\sum\limits_{j=0}^{n-2} \binom{n-2}{j}p^*L^jq^*M^{n-2-j} \\ & =q^*M^{n-2}+(n-2)p^*Lq^*M^{n-3}+\binom{n-2}{2}p^*L^2q^*M^{n-4}.
\end{aligned}
\end{equation}
Now, $K_X=p^*K_S+q^*K_Y$. If $n \ge 4$, using \eqref{hn-2}, we get
\begin{equation}
\label{k4}
\begin{aligned}
K_X^2H^{n-2} & = \big(p^*K_S^2+q^*K_Y^2+2p^*K_Sq^*K_Y\big)\big(q^*M^{n-2}+(n-2)p^*Lq^*M^{n-3}+\binom{n-2}{2}p^*L^2q^*M^{n-4}\big) \\
& = K_S^2M^{n-2}+\binom{n-2}{2}L^2(K_Y^2M^{n-4})+2(n-2)(LK_S)(K_YM^{n-3}).
\end{aligned}
\end{equation}
If $n=3$, the same calculation shows that
\begin{equation}
\label{k3}
K_X^2H=K_S^2\deg(M)+2(LK_S)\deg(K_Y).
\end{equation}
Next, $T_X \cong p^*T_S \oplus q^*T_Y$ and therefore
$$c_2(X)=p^*c_2(S)+p^*K_Sq^*K_Y+q^*c_2(Y).$$
Using again \eqref{hn-2} when $n \ge 4$, we find 
\begin{equation}
\label{c4}
\begin{aligned}
c_2(X)H^{n-2} & = \big(p^*c_2(S)+p^*K_Sq^*K_Y+q^*c_2(Y)\big)\big(q^*M^{n-2}+(n-2)p^*Lq^*M^{n-3}+\binom{n-2}{2}p^*L^2q^*M^{n-4}\big) \\
& = c_2(S)M^{n-2}+(n-2)(LK_S)(K_YM^{n-3})+\binom{n-2}{2}L^2(c_2(Y)M^{n-4}).
\end{aligned}
\end{equation}
When $n=3$ we get
\begin{equation}
\label{c3}
\begin{aligned}
c_2(X)H & = \big(p^*c_2(S)+p^*K_Sq^*K_Y\big)\big(p^*L+q^*M\big) \\
& = c_2(S) \deg(M)+(LK_S) \deg(K_Y).
\end{aligned}
\end{equation}

We now consider \eqref{dis}. 

First, assume that $n \ge 4$. Using \eqref{h}, \eqref{k4} and \eqref{c4}, we see that \eqref{dis} is holds if and only if
$$\begin{aligned} 
(n+1)\binom{n}{2}L^2 M^{n-2} & < K_S^2M^{n-2}+\binom{n-2}{2}L^2(K_Y^2M^{n-4})+2(n-2)(LK_S)(K_YM^{n-3}) \\
& -2c_2(S)M^{n-2}-2(n-2)(LK_S)(K_YM^{n-3})-2\binom{n-2}{2}L^2(c_2(Y)M^{n-4})
\end{aligned}$$ 
if and only if
\begin{equation}
\label{f1}
\big((n+1)\binom{n}{2}M^{n-2}-\binom{n-2}{2}K_Y^2M^{n-4}+2\binom{n-2}{2}c_2(Y)M^{n-4}\big)L^2 < M^{n-2}\big(K_S^2-2c_2(S)\big).
\end{equation}
Set
$$f(n,Y,M)=(n+1)\binom{n}{2}M^{n-2}-\binom{n-2}{2}K_Y^2M^{n-4}+2\binom{n-2}{2}c_2(Y)M^{n-4}$$
so that \eqref{f1} becomes
\begin{equation}
\label{f2}
f(n,Y,M)L^2 < M^{n-2}\big(K_S^2-2c_2(S)\big).
\end{equation}
By \cite[Thm.~1.2]{a} we have that $L^2=7k^9$ and $K_S^2-2c_2(S)=3\sigma(S)=3(3k^2-11)k^9>0$.
Therefore, if $f(n,Y,M) \le 0$, then certainly \eqref{f2} holds. On the other hand, if $f(n,Y,M)>0$,  \eqref{f2} is equivalent to
$$\frac{L^2}{\sigma(S)}=\frac{7}{3k^2-11} <\frac{3M^{n-2}}{f(n,Y,M)}$$
that certainly holds for $k \gg 0$.

Finally, assume that $n=3$. Using \eqref{h}, \eqref{k3} and \eqref{c3}, we see that \eqref{dis} is holds if and only if
$$\begin{aligned} 
12L^2 \deg(M) & < K_S^2\deg(M)+2(LK_S)\deg(K_Y) \\
& -2c_2(S)\deg(M)-2(LK_S)\deg(K_Y)
\end{aligned}$$ 
if and only if
$$L^2 < \frac{K_S^2-2c_2(S)}{12}=\frac{\sigma(S)}{4}$$
that is
\begin{equation}
\label{f23}
\frac{L^2}{\sigma(S)}=\frac{7}{3k^2-11} <\frac{1}{4}
\end{equation}
that certainly holds for $k \ge 4$.

Therefore, in all cases, \eqref{dis} is satisfied and we conclude by Lemma \ref{nonext} that $(X, H)$ and all its successive hyperplane sections until dimension $2$, do not carry an Ulrich vector bundle.

As for $(X', H')$, the plan is to check again \eqref{dis}.

We have
\begin{equation}
\label{11}
(H')^{n-1}=aH^n
\end{equation}
and
\begin{equation}
\label{12}
K_{X'}^2(H')^{n-3}=aK_X^2H^{n-2}+2a^2K_XH^{n-1}+a^3H^n.
\end{equation}
From the exact sequence
$$0 \to T_{X'} \to {T_X}_{|X'} \to \O_{X'}(aH') \to 0$$
we obtain $c_2(X')=\big(c_2(X)+aK_XH+a^2H^2\big)_{|X'}$, and consequently,
\begin{equation}
\label{13}
c_2(X')(H')^{n-3}=ac_2(X)H^{n-2}+a^2K_XH^{n-1}+a^3H^n.
\end{equation}
Now, using \eqref{11}, \eqref{12} and \eqref{13}, we see that \eqref{dis} for $(X', H')$, that is
$$n(H')^{n-1}<(K_{X'}^2-2c_2(X'))(H')^{n-3}$$
is equivalent to
$$naH^n<aK_X^2H^{n-2}-2ac_2(X)H^{n-2}-a^3H^n$$
hence to
\begin{equation}
\label{14}
(n+a^2)H^n<\big(K_X^2-2c_2(X)\big)H^{n-2}.
\end{equation}
The same calculations as the case $(X,H)$ show that \eqref{14} is the same, for $n \ge 4$, as \eqref{f2} with the new function
$$f(n,a,Y,M)=(n+a^2)\binom{n}{2}M^{n-2}-\binom{n-2}{2}K_Y^2M^{n-4}+2\binom{n-2}{2}c_2(Y)M^{n-4}.$$
That is, \eqref{14} is equivalent to
$$f(n,a,Y,M)L^2 < M^{n-2}\big(K_S^2-2c_2(S)\big)$$
and the latter is certainly satisfied for $k \gg 0$ as we showed for \eqref{f2}. Similarly, when $n=3$, the same calculations as the case $(X,H)$ show that \eqref{14} is similar to \eqref{f23}, namely is
$$\frac{L^2}{\sigma(S)}=\frac{7}{3k^2-11} <\frac{1}{3+a^2}$$
that is certainly satisfied for $k \gg 0$. We therefore we conclude by again by Lemma \ref{nonext}.

It remains to show that $X'$ is not a product. The latter follows from Lemma \ref{nopro} when $n \ge 4$. If $n=3$, assume that $X' \cong C_1 \times C_2$, for two smooth curves $C_i, i \in \{1,2\}$. Let $p_i : D=Y_1 \times Y_2 \to Y_i$ be the projections, so that $T_D \cong p_1^*T_{Y_1} \oplus p_2^*T_{Y_2}$. As in the proof of Lemma \ref{nopro}, we have that $p_2(D)=p_2(Y)=0$, while $p_2(X)=\pi_1^*p_2(S)$. Therefore, the normal bundle sequence of $X'$ gives, using \eqref{h},
$$\begin{aligned}
3a^3L^2\deg(M) & =a^3H^3=p_2(\O_{X'}(aH))=p_2(X)_{|X'}=a\pi_1^*p_2(S)H= a\big(\pi_1^*c_1(S)^2-2\pi_1^*c_2(S)\big)(\pi_1^*L+\pi_2^*M\big) \\
& = a\big(c_1(S)^2-2c_2(S)\big)\deg(M)=3a\sigma(S)\deg(M)
\end{aligned}$$
and dividing by $3a\deg(M)$, the above becomes
$$a^2L^2=\sigma(S)$$
that is 
$$7a^2=3k^2-11$$
hence $3k^2 \equiv 4$ (mod $7$), a contradiction. Therefore $X'$ is not a product. 
\end{proof}
\renewcommand{\proofname}{Proof}

{\bf Acknowledgement}

We thank R. Vacca for a couple of useful remarks on a previous version.

\end{document}